\documentclass[12pt]{amsart}

\newtheorem{theorem}{Theorem}[section]
\newtheorem{lemma}[theorem]{Lemma}

\theoremstyle{definition}

\theoremstyle{remark}
\newtheorem{remark}[theorem]{Remark}

\numberwithin{equation}{section}

\let \la=\lambda
\let \e=\varepsilon
\let \d=\delta
\let \o=\omega
\let \a=\alpha
\let \f=\varphi
\let \b=\beta

\let \g=\gamma
\let \O=\Omega

\let \G=\Gamma
\let \ga=\gamma

\begin{document}
\title[Sharp weighted norm inequalities]
{Sharp weighted norm inequalities for Littlewood-Paley operators and
singular integrals}

\author{Andrei K. Lerner}
\address{Department of Mathematics,
Bar-Ilan University, 52900 Ramat Gan, Israel}
\email{aklerner@netvision.net.il}

\begin{abstract}
We prove sharp $L^p(w)$ norm inequalities for the intrinsic square
function (introduced recently by M. Wilson) in terms of the $A_p$
characteristic of $w$ for all $1<p<\infty$. This implies the same
sharp inequalities for the classical Lusin area integral $S(f)$, the
Littlewood-Paley $g$-function, and their continuous analogs
$S_{\psi}$ and $g_{\psi}$. Also, as a corollary, we obtain sharp
weighted inequalities for any convolution Calder\'on-Zygmund
operator for all $1<p\le 3/2$ and $3\le p<\infty$, and for its
maximal truncations for $3\le p<\infty$.
\end{abstract}



\maketitle

\section{Introduction}
Given a weight (i.e., a non-negative locally integrable function)
$w$, its $A_p, 1<p<\infty,$ characteristic is defined by
$$\|w\|_{A_p}\equiv\sup_Q\left(\frac{1}{|Q|}\int_Qw\,dx\right)
\left(\frac{1}{|Q|}\int_Qw^{-\frac{1}{p-1}}\,dx\right)^{p-1},$$
where the supremum is taken over all cubes $Q\subset {\mathbb R}^n$.

The main conjecture (which is implicit in  work of Buckley \cite{B})
concerning the behavior of singular integrals $T$ on $L^p(w)$ says
that
\begin{equation}\label{sing}
\|T\|_{L^p(w)}\le
c(T,n,p)\|w\|_{A_p}^{\max\big(1,\frac{1}{p-1}\big)}\quad(1<p<\infty).
\end{equation}

For Littlewood-Paley operators $S$ (we specify below the class of
such operators we shall deal with) it was conjectured in \cite{L1}
that
\begin{equation}\label{lp}
\|S\|_{L^p(w)}\le
c(S,n,p)\|w\|_{A_p}^{\max\big(\frac{1}{2},\frac{1}{p-1}\big)}\quad(1<p<\infty).
\end{equation}

Observe that the exponents $\max\big(1,\frac{1}{p-1}\big)$ in
(\ref{sing}) and $\max\big(\frac{1}{2},\frac{1}{p-1}\big)$ in
(\ref{lp}) are best possible for all $1<p<\infty$ (see
\cite{B,L1,L2}). Also, by the sharp version of the Rubio de Francia
extrapolation theorem \cite{DGPP}, inequality (\ref{sing}) for $p=2$
implies (\ref{sing}) for all $p>1$; analogously, it is enough to
prove (\ref{lp}) for $p=3$.

Currently conjecture (\ref{sing}) is proved for: the Hilbert
transform and Riesz transforms (Petermichl \cite{P1,P2}), the
Ahlfors-Beurling operator (Petermichl and Volberg \cite{PV}), any
one-dimensional Calder\'on-Zygmund convolution operator with
sufficiently smooth kernel (Vagharshakyan \cite{V}). The proofs in
\cite{P1,P2,PV} are based on the so-called Haar shift operators
combined with the Bellman function technique. The main idea in
\cite{V} is also based on Haar shifts.

Recently, Lacey, Petermichl and Reguera \cite{LPR} have established
sharp weighted estimates for Haar shift operators without the use of
Bellman functions; their proof uses a two-weight ``$Tb$ theorem" for
Haar shift operators due to Nazarov, Treil and Volberg \cite{NTV}.
This provides a unified approach to works \cite{P1,P2,PV,V}. Very
recently, a new, more elementary proof of this result, avoiding the
$Tb$ theorem, was given by Cruz-Uribe, Martell and P\'erez
\cite{CMP}; a key ingredient in \cite{CMP} was a decomposition of an
arbitrary measurable function in terms of local mean oscillations
obtained in \cite{L3}.

It was shown in \cite{CMP} that such a decomposition is very
convenient when dealing with certain dyadic type operators. In
particular, using these ideas, the authors proved in \cite{CMP}
conjecture (\ref{lp}) for the dyadic square function. Note that this
is the first result establishing (\ref{lp}) for all $p>1$.
Previously (\ref{lp}) was obtained for the dyadic square function in
the case $p=2$ by Hukovic, Treil and Volberg~\cite{HTV}, and,
independently by Wittwer \cite{Wit1}. Also, (\ref{lp}) in the case
$p=2$ was proved by Wittwer~\cite{Wit2} for the continuous square
function. By the extrapolation argument \cite{DGPP}, the linear
$\|w\|_{A_2}$ bound implies the bound by
$\|w\|_{A_p}^{\max(1,1/(p-1))}$ for all $p>1$. However, for $p>2$
this is not sharp for square functions. In~ \cite{L2}, the linear
bound for $p>2$ was improved to $\|w\|_{A_p}^{p'/2}$ for a large
class of Littlewood-Paley operators.

In this paper we show that similar arguments to those developed in
~\cite{CMP} work actually for essentially any important
Littlewood-Paley operator. To be more precise, we prove conjecture
(\ref{lp}) for the so-called intrinsic square function $G_{\a}$
introduced by Wilson \cite{Wi2}. As it was shown in \cite{Wi2}, the
intrinsic square function pointwise dominates both classical square
functions and their more recent analogs. As a result, we have the
following theorem.

\begin{theorem}\label{LP} Conjecture (\ref{lp}) holds for any one
of the following operators: the intrinsic square function
$G_{\a}(f)$, the Lusin area integral~$S(f)$, the Littlewood-Paley
function $g(f)$, the continuous square functions $S_{\psi}(f)$ and
$g_{\psi}(f)$.
\end{theorem}

In Section 2 below we give precise definitions of the operators
appeared in Theorem \ref{LP}.

We mention briefly the main difference between the proof of Theorem~
\ref{LP} and the corresponding proof for the dyadic square function
in~\cite{CMP}. Let ${\mathcal D}$ be the set of all dyadic cubes in
${\mathbb R}^n$. Dealing with the dyadic square function, we arrive
to the mean oscillation of the sum $\sum_{Q\in {\mathcal
D}}\xi_{Q}(f)\chi_Q$ on any dyadic cube $Q_0$. The corresponding
object can be easily handled because of the nice interaction between
any two dyadic cubes. Now, working with the intrinsic square
function, we have to estimate the mean oscillation of the sum
$\sum_{Q\in {\mathcal D}}\xi'_{Q}(f)\chi_{3Q}$ on any dyadic cube
$Q_0$. Here we use several tricks. First, as it was shown by Wilson
\cite{Wi1}, the set ${\mathcal D}$ can be divided into $3^n$
disjoint families ${\mathcal D}_k$ such that the cubes $\{3Q:Q\in
{\mathcal D}_k\}$ behave essentially as the dyadic cubes. Second,
given any dyadic cube $Q_0$, one can find in each family ${\mathcal
D}_k$ the cube $Q_k$ such that $Q_0\subset 3Q_k\subset 5Q_0$. This
is proved in Lemma~\ref{prop} below. Combining these tricks, we
arrive to exactly the same situation as described above for the
dyadic square function.

The Littlewood-Paley technique developed by Wilson in
\cite{Wi1,Wi2,Wi3} along with Theorem \ref{LP} allows us to get
conjecture (\ref{sing}) for classical Calder\'on-Zygmund operators
for any $p\in (1,3/2]\cup[3,\infty)$. For
$K_{\d}(x)=K(x)\chi_{\{|x|>\d\}}$ let
$$
Tf(x)=\lim_{\d\to 0}f*K_{\d}(x)\quad\text{and}\quad
T^*f(x)=\sup_{\d>0}|f*K_{\d}(x)|,
$$
where the kernel $K$ satisfies the standard conditions:
$$
|K(x)|\le\frac{c}{|x|^n}, \quad \int_{r<|x|<R}K(x)\,dx=0
\quad(0<r<R<\infty),
$$
and
$$
|K(x)-K(x-y)|\le\frac{c|y|^{\e}}{|x|^{n+\e}}\quad(|y|\le
|x|/2,\e>0).
$$

\begin{theorem}\label{SIN}
Conjecture (\ref{sing}) holds for $T$ for any $1<p\le 3/2$ and $3\le
p<\infty$. Also, (\ref{sing}) holds for $T^*$ for any $p\ge 3$.
\end{theorem}

Notice that for the maximal Hilbert, Riesz and Ahlfors-Beurling
transforms conjecture (\ref{sing}) was recently proved for any $p>1$
by Hyt\"onen {\it et. al.} \cite{HLRV}; a different proof for the
same operators is given in \cite{CMP}.

The proof of Theorem \ref{SIN} is based essentially on Theorem
\ref{LP}, on the pointwise estimate $S_{\psi}(Tf)(x)\le
cG_{\a}(f)(x)$ (proved in \cite{Wi2,Wi3}), and on a version of the
Chang-Wilson-Wolff theorem \cite{CWW} proved in \cite{Wi1}.

\section{Littlewood-Paley operators}
Let ${\mathbb R}^{n+1}_+={\mathbb R}^{n}\times{\mathbb R}_{+}$ and
$\G_{\b}(x)=\{(y,t)\in {\mathbb{R}}^{n+1}_+:|y-x|<\b t\}$. Here and
below we drop the subscript $\b$ if $\b=1$. Set
$\f_t(x)=t^{-n}\f(x/t)$.

The classical square functions are defined as follows. If
$u(x,t)=P_t*f(x)$ is the Poisson integral of $f$, the Lusin area
integral $S_{\b}$ and the Littlewood-Paley $g$-function are defined
respectively by
$$
S_{\b}(f)(x)=\left(\int_{\G_{\b}(x)}|\nabla
u(y,t)|^2\frac{dydt}{t^{n-1}}\right)^{1/2}
$$
and
$$
g(f)(x)=\left(\int_0^{\infty}t|\nabla u(x,t)|^2dt\right)^{1/2}.
$$

The modern (real-variable) variants of $S_{\b}$ and $g$ can be
defined in the following way. Let $\psi\in C^{\infty}({\mathbb
R}^n)$ be radial, supported in $\{x:|x|\le 1\}$, and $\int\psi=0$.
The continuous square functions $S_{\psi,\b}$ and $g_{\psi}$ are
defined by
$$S_{\psi,\b}(f)(x)=\left(\int_{\G_{\b}(x)}|f*\psi_t(y)|^2
\frac{dydt}{t^{n+1}}\right)^{1/2}$$ and
$$
g_{\psi}(f)(x)=\left(\int_0^{\infty}|f*\psi_t(x)|^2
\frac{dt}{t}\right)^{1/2}.
$$

In \cite{Wi2} (see also \cite[p. 103]{Wi3}), it was introduced a new
square function which is universal in a sense. This function is
independent of any particular kernel $\psi$, and it dominates
pointwise all the above defined square function. On the other hand,
it is not essentially larger than any particular $S_{\psi,\b}(f)$.
For $0<\a\le 1$, let ${\mathcal C}_{\a}$ be the family of functions
supported in $\{x:|x|\le 1\}$, satisfying $\int\psi=0$, and such
that for all $x$ and $x'$, $|\f(x)-\f(x')|\le |x-x'|^{\a}$. If $f\in
L^1_{\text{loc}}({\mathbb R}^n)$ and $(y,t)\in {\mathbb R}^{n+1}_+$,
we define
$$
A_{\a}(f)(y,t)=\sup_{\f\in {\mathcal C}_{\a}}|f*\f_t(y)|.
$$
The intrinsic square function is defined by
$$
G_{\b,\a}(f)(x)=\left(\int_{\G_{\b}(x)}
\big(A_{\a}(f)(y,t)\big)^2\frac{dydt}{t^{n+1}}\right)^{1/2}.
$$
If $\b=1$, set $G_{1,\a}(f)=G_{\a}(f)$.

We mention several properties of $G_{\a}(f)$ (for the proofs we
refer to \cite{Wi2} and \cite[Ch. 6]{Wi3}). First of all, it is of
weak type $(1,1)$:
\begin{equation}\label{weak}
|\{x\in {\mathbb R}^n:G_{\a}(f)(x)>\la\}|\le
\frac{c(n,\a)}{\la}\int_{{\mathbb R}^n}|f|\,dx.
\end{equation}
Second, if $\b\ge 1$, then for all $x\in {\mathbb R}^n$,
\begin{equation}\label{estG}
G_{\b,\a}(f)(x)\le c(\a,\b,n)G_{\a}(f)(x).
\end{equation}
Third, if $S$ is anyone of the Littlewood-Paley operators defined
above, then
\begin{equation}\label{poest}
S(f)(x)\le cG_{\a}(f)(x),
\end{equation}
where the constant $c$ is independent of $f$ and $x$.

\section{Dyadic cubes}

We say that $I\subset {\mathbb R}$ is a dyadic interval if $I$ is of
the form $(\frac{j}{2^k},\frac{j+1}{2^k})$ for some integers $j$ and
$k$. A dyadic cube $Q\subset {\mathbb R}^n$ is a Cartesian product
of $n$ dyadic intervals of equal lengths. Let ${\mathcal D}$ be the
set of all dyadic cubes in ${\mathbb R}^n$.

Denote by $\ell_Q$ the side length of $Q$. Given $r>0$, let $rQ$ be
the cube with the same center as $Q$ such that $\ell_{rQ}=r\ell_Q$.

The following result can be found in \cite[Lemma 2.1]{Wi1} or in
\cite[p.~91]{Wi3}.

\begin{lemma}\label{wil} There exist disjoint families ${\mathcal
D}_1,\dots {\mathcal D}_{3^n}$ of dyadic cubes such that ${\mathcal
D}=\cup_{k=1}^{3^n}{\mathcal D}_k$, and, for every $k$, if $Q_1,Q_2$
are in ${\mathcal D}_k$, then $3Q_1$ and $3Q_2$ are either disjoint
or one is contained in the other.
\end{lemma}

Observe that it suffices to prove the lemma in the one-dimensional
case. Indeed, if ${\mathcal I}$ is the set of all dyadic intervals
in ${\mathbb R}$ and ${\mathcal I}=\cup_{j=1}^3{\mathcal I}_j$ is
the representation from Lemma \ref{wil} in the case $n=1$, then the
required families in ${\mathbb R}^n$ are of the form
$${\mathcal D}_k=\Big\{\prod_{m=1}^nI_m:I_m\in {\mathcal
I}_{\a_i},\a_i\in \{1,2,3\}\Big\}\quad(k=1,\dots,3^n).$$

The following property of the families ${\mathcal D}_k$ will play an
important role in the proof of Theorem \ref{LP}.

\begin{lemma}\label{prop} For any cube $Q\in {\mathcal D}$
and for each $k=1,\dots,3^n$ there is a cube $Q_k\in {\mathcal D}_k$
such that $Q\subset 3Q_k\subset 5Q$.
\end{lemma}

\begin{proof}
Let us consider first the one-dimensional case. Assume that
${\mathcal I}=\cup_{j=1}^3{\mathcal I}_j$ is the representation from
Lemma \ref{wil}.

Take an arbitrary dyadic interval
$J=(\frac{j}{2^k},\frac{j+1}{2^k})$. Set $J_1=J$. Consider the
dyadic intervals $J_2=(\frac{j-1}{2^k},\frac{j}{2^k})$ and
$J_3=(\frac{j+1}{2^k},\frac{j+2}{2^k})$. It is easy to see that each
two different intervals of the form $3J_l$  are neither disjoint nor
one is contained in the other. Therefore, the intervals $J_l$ lie in
the different families ${\mathcal I}_j$. Also, $J\subset 3J_l\subset
5J$ for $l=1,2,3$.

Consider now the multidimensional case. Take an arbitrary cube $Q\in
{\mathcal D}$. Then $Q=\prod_{m=1}^nI_m$, where $I_m\in {\mathcal
I}$ and $\ell_{I_m}=h$ for each $m$. Fix $\a_i\in \{1,2,3\}$. We
have already proved that there exists $\widetilde I_m\in {\mathcal
I}_{\a_i}$ such that $I_m\subset 3\widetilde I_m\subset 5I_m$.
Observe also that, by the one-dimensional construction,
$\ell_{\widetilde I_m}=\ell_{I_m}=h$. Therefore, setting
$Q_k=\prod_{m=1}^n\widetilde I_m$, we obtain the required cube from
${\mathcal D}_k$.
\end{proof}

\section{Local mean oscillations}
Given a measurable function $f$ on ${\mathbb R}^n$ and a cube $Q$,
define the local mean oscillation of $f$ on $Q$ by
$$\o_{\la}(f;Q)=\inf_{c\in {\mathbb R}}
\big((f-c)\chi_{Q}\big)^*\big(\la|Q|\big)\quad(0<\la<1),$$ where
$f^*$ denotes the non-increasing rearrangement of $f$. The local
sharp maximal function relative to $Q$ is defined by
$$M^{\#}_{\la;Q}f(x)=\sup_{x\in Q'\subset
Q}\o_{\la}(f;Q'),$$ where the supremum is taken over all cubes
$Q'\subset Q$ containing the point $x$.

By a median value of $f$ over $Q$ we mean a possibly nonunique, real
number $m_f(Q)$ such that
$$\max\big(|\{x\in Q: f(x)>m_f(Q)\}|,|\{x\in Q: f(x)<m_f(Q)\}|\big)\le |Q|/2.$$
It follows from the definition that
\begin{equation}\label{infty}
|m_f(Q)|\le (f\chi_Q)^*(|Q|/2).
\end{equation}

Given a cube $Q_0$, denote by ${\mathcal D}(Q_0)$ the set of all
dyadic cubes with respect to $Q_0$ (that is, they are formed by
repeated subdivision of $Q_0$ and each of its descendants into $2^n$
congruent subcubes). Observe that if $Q_0\in {\mathcal D}$, then the
cubes from ${\mathcal D}(Q_0)$ are also dyadic in the usual sense as
defined in the previous section.

If $Q\in {\mathcal D}(Q_0)$ and $Q\not=Q_0$, we denote by $\widehat
Q$ its dyadic parent, that is, the unique cube from ${\mathcal
D}(Q_0)$ containing $Q$ and such that $|\widehat Q|=2^n|Q|$.

The following result has been recently proved in \cite{L3}.

\begin{theorem}\label{main} Let $f$ be a measurable function on
${\mathbb R}^n$ and let $Q_0$ be a fixed cube. Then there exists a
(possibly empty) collection of cubes $Q_j^k\in {\mathcal D}(Q_0)$
such that
\begin{enumerate}
\renewcommand{\labelenumi}{(\roman{enumi})}
\item
for a.e. $x\in Q_0$,
$$
|f(x)-m_f(Q_0)|\le 4M_{1/4;Q_0}^{\#}f(x)+4\sum_{k=1}^{\infty}\sum_j
\o_{\frac{1}{2^{n+2}}}(f;\widehat Q_j^k)\chi_{Q_j^k}(x);
$$
\item
for each fixed $k$ the cubes $Q_j^k$ are pairwise disjoint;
\item
if $\O_k=\cup_jQ_j^k$, then $\O_{k+1}\subset \O_k$;
\item
$|\O_{k+1}\cap Q_j^k|\le \frac{1}{2}|Q_j^k|.$
\end{enumerate}
\end{theorem}

\begin{remark}\label{rem}
The proof of Theorem \ref{main} shows that actually
$M_{1/4;Q_0}^{\#}f$ can be replaced by a smaller dyadic operator,
that is, by
$$
M_{1/4;Q_0}^{\#,d}f(x)=\sup_{Q\ni x, Q\in {\mathcal
D}(Q_0)}\o_{1/4}(f;Q).
$$
\end{remark}
Note that Theorem \ref{main} is a development of ideas going back to
works of Carleson \cite{C}, Garnett-Jones \cite{GJ} and Fujii
\cite{F}.

We mention a simple property of local mean oscillations which will
be used below.

\begin{lemma}\label{simpr} For any $k\in {\mathbb N}$ and for each
cube $Q$,
\begin{equation}\label{finsum}
\o_{\la}\Big(\sum_{i=1}^kf_i;Q\Big)\le
\sum_{i=1}^k\o_{\la/k}(f_i;Q)\quad(0<\la<1).
\end{equation}

\begin{proof}
It is well known (see, e.g., \cite[p. 41]{BS}) that
$$(f+g)^*(t_1+t_2)\le f^*(t_1)+g^*(t_2)\quad(t_1,t_2>0).$$
This property is easily extended to any finite sum of functions:
$$\Big(\sum_{i=1}^kf_i\Big)^*(t)\le \sum_{i=1}^k(f_i)^*(t/k).$$
Hence, for arbitrary $\xi_i\in {\mathbb R}$ we have
\begin{eqnarray*}
\o_{\la}\Big(\sum_{i=1}^kf_i;Q\Big)&=&\inf_{c\in {\mathbb
R}}\Big(\big(\sum_{i=1}^kf_i-c\big)\chi_Q\Big)^*(\la|Q|)\\
&=&\inf_{c\in {\mathbb
R}}\Big(\big(f_1-c+\sum_{i=2}^k(f_i-\xi_i)\big)\chi_Q\Big)^*(\la|Q|)\\
&\le& \o_{\la/k}(f_1;Q)+\sum_{i=2}^k((f_i-\xi_i)\chi_Q)^*(\la|Q|/k).
\end{eqnarray*}
Taking the infimum over all $\xi_i\in {\mathbb R}$ yields
(\ref{finsum}).
\end{proof}
\end{lemma}

\section{Proof of Theorems \ref{LP} and \ref{SIN}}
\subsection{The intrinsic square function $\widetilde G_{\a}$}
For our purposes it will be more convenient to work with the
following variant of $G_{\a}$. Given a cube $Q\subset {\mathbb
R}^n$, set
$$T(Q)=\{(y,t)\in {\mathbb R}^n:y\in Q, \ell(Q)/2\le t<\ell(Q)\}.$$
Denote $\g_Q(f)^2=\int_{T(Q)}
\big(A_{\a}(f)(y,t)\big)^2\frac{dydt}{t^{n+1}}$ and let
$$
\widetilde G_{\a}(f)(x)^2=\sum_{Q\in {\mathcal
D}}\g_Q(f)^2\chi_{3Q}(x).
$$

\begin{lemma}\label{point} For any $x\in {\mathbb R}^n$,
\begin{equation}\label{poi}
G_{\a}(f)(x)\le \widetilde G_{\a}(f)(x)\le c(\a,n)G_{\a}(f)(x).
\end{equation}
\end{lemma}

\begin{proof}
For any $x\not\in 3Q$ we have $\G(x)\cap T(Q)=\emptyset$, and hence
$$
\int_{\G(x)\cap T(Q)}
\big(A_{\a}(f)(y,t)\big)^2\frac{dydt}{t^{n+1}}\le
\g_Q(f)^2\chi_{3Q}(x).
$$
Therefore,
\begin{eqnarray*}
G_{\a}(f)(x)^2&=&\int_{\G(x)}
\big(A_{\a}(f)(y,t)\big)^2\frac{dydt}{t^{n+1}}\\
&=&\sum_{Q\in{\mathcal D}}\int_{\G(x)\cap T(Q)}
\big(A_{\a}(f)(y,t)\big)^2\frac{dydt}{t^{n+1}}\le \widetilde
G_{\a}(f)(x)^2.
\end{eqnarray*}

On the other hand, if $x\in 3Q$ and $(y,t)\in T(Q)$, then $|x-y|\le
2\sqrt n\ell(Q)\le 4\sqrt n t$. Thus,
\begin{eqnarray*}
&&\widetilde G_{\a}(f)(x)^2=\sum_{Q\in {\mathcal
D}}\g_Q(f)^2\chi_{3Q}(x)\\
&&\le \sum_{Q\in {\mathcal D}}\int_{T(Q)\cap \G_{4\sqrt n}(x)}
\big(A_{\a}(f)(y,t)\big)^2\frac{dydt}{t^{n+1}}=G_{4\sqrt
n,\a}(f)(x)^2.
\end{eqnarray*}
Combining this with (\ref{estG}), we get the right-hand side of
(\ref{poi}).
\end{proof}

\subsection{A local mean oscillation estimate of $\widetilde
G_{\a}$} The key role in our proof will be played by the following
lemma.

\begin{lemma}\label{key}
For any cube $Q\in {\mathcal D}$,
$$
\o_{\la}(\widetilde G_{\a}(f)^2;Q)\le
c(n,\a,\la)\left(\frac{1}{|15Q|}\int_{15Q}|f|dx\right)^2.
$$
\end{lemma}

\begin{proof}
Applying Lemma \ref{wil}, we can write
$$
\widetilde G_{\a}(f)(x)^2=\sum_{k=1}^{3^n}\sum_{Q\in {\mathcal
D}_k}\g_Q(f)^2\chi_{3Q}(x)\equiv \sum_{k=1}^{3^n}\widetilde
G_{\a,k}(f)(x)^2.
$$
Hence, by Lemma \ref{simpr},
$$
\o_{\la}(\widetilde G_{\a}(f)^2;Q)\le\sum_{k=1}^{3^n}
\o_{\la/3^n}(\widetilde G_{\a,k}(f)^2;Q).
$$

By Lemma \ref{prop}, for each $k=1,\dots,3^n$ there exists a cube
$Q_k\in {\mathcal D}_k$ such that $Q\subset 3Q_k\subset 5Q$. Hence,
\begin{eqnarray*}
&&\inf_{c\in {\mathbb R}}\Big(\big(\widetilde
G_{\a,k}(f)^2-c\big)\chi_Q\Big)^*\big(\la|Q|/3^n\big)\\
&&\le \inf_{c\in {\mathbb R}}\Big(\big(\widetilde
G_{\a,k}(f)^2-c\big)\chi_{3Q_k}\Big)^*\big(\la|Q|/3^n\big).
\end{eqnarray*}
Using the main property of cubes from the family ${\mathcal D}_k$
(expressed in Lemma~\ref{wil}), for any $x\in 3Q_k$ we have
\begin{equation}\label{est}
\widetilde G_{\a,k}(f)(x)^2=\sum_{Q\in {\mathcal D}_k:3Q\subset
3Q_k}\g_Q(f)^2\chi_{3Q}(x)+\sum_{Q\in {\mathcal D}_k:3Q_k\subset
3Q}\g_Q(f)^2.
\end{equation}

Arguing as in the proof of Lemma \ref{point}, we obtain
\begin{eqnarray*}
&&\sum_{Q\in {\mathcal D}_k:3Q\subset
3Q_k}\g_Q(f)^2\chi_{3Q}(x)\\
&&\le \sum_{Q\in {\mathcal D}_k:3Q\subset 3Q_k} \int_{T(Q)\cap
\G_{4\sqrt n}(x)} \big(A_{\a}(f)(y,t)\big)^2\frac{dydt}{t^{n+1}}\\
&&\le \int_{\widehat T(3Q_k)\cap \G_{4\sqrt n}(x)}
\big(A_{\a}(f)(y,t)\big)^2\frac{dydt}{t^{n+1}},
\end{eqnarray*}
where $\widehat T(3Q_k)=\{(y,t):y\in 3Q_k,0<t\le \ell(3Q_k)\}$. For
any $\f$ supported in $\{x:|x|\le 1\}$ and for $(y,t)\in \widehat
T(3Q_k)$ we have
$$f*\f_t(y)=(f\chi_{9Q_k})*\f_t(y).$$
Therefore,
$$
\int_{\widehat T(3Q_k)\cap \G_{4\sqrt n}(x)}
\big(A_{\a}(f)(y,t)\big)^2\frac{dydt}{t^{n+1}}\le G_{4\sqrt
n,\a}(f\chi_{9Q_k})(x)^2.
$$
Combining the letter estimates with (\ref{est}) and setting
$$c=\sum_{Q\in {\mathcal D}_k:3Q_k\subset 3Q}\g_Q(f)^2,$$
we get
$$0\le \widetilde G_{\a,k}(f)(x)^2-c\le G_{4\sqrt
n,\a}(f\chi_{9Q_k})(x)^2\quad(x\in 3Q_k).$$ From this, by
(\ref{weak}) and (\ref{estG}) (we use also that $3Q_k\subset 5Q$
implies $9Q_k\subset 15Q$),
\begin{eqnarray*} &&\inf_{c\in
{\mathbb R}}\Big(\big(\widetilde
G_{\a,k}(f)^2-c\big)\chi_{3Q_k}\Big)^*\big(\la|Q|/3^n\big)\\
&&\le c(n,\a)\big(G_{\a}(f\chi_{9Q_k})\big)^*(\la|Q|/3^n)^2\\
&&\le c\left(\frac{3^n}{\la|Q|}\int_{9Q_k}|f|\right)^2 \le
c\left(\frac{3^n}{\la|Q|}\int_{15Q}|f|\right)^2,
\end{eqnarray*}
which completes the proof.
\end{proof}

\subsection{An auxiliary operator} Let $\{Q_j^k\}$ be a family
of cubes appeared in Theorem \ref{main}.

Given $\ga>1$, consider the operator ${\mathcal A}_{\ga}$ defined by
$$
{\mathcal A}_{\ga}f(x)=\sum_{k,j}\left(\frac{1}{|\ga
Q_j^k|}\int_{\ga Q_j^k}|f|\right)^2\chi_{Q_j^k}(x).
$$
This object will appear naturally after the combination of Lemma
\ref{key} with Theorem \ref{main}.

\begin{lemma}\label{estag} For any $f\in L^3(w)$,
\begin{equation}\label{ag}
\left(\int_{Q_{0}}({\mathcal A}_{\ga}f)^{3/2}w\,dx\right)^{2/3}\le
c(n,\ga)\|w\|_{A_3}\|f\|_{L^3(w)}^2.
\end{equation}
\end{lemma}

\begin{proof} We follow \cite{CMP} with some minor modifications.
By duality, (\ref{ag}) is equivalent to that for any $h\ge 0$ with
$\|h\|_{L^3(w)}=1$,
\begin{eqnarray}
\int_{Q_0}({\mathcal
A}f)hw\,dx&=&\sum_{k,j}\Big(\frac{1}{|\ga Q_j^k|}\int_{\ga Q_j^k}|f|\Big)^2\int_{Q_j^k}hw\label{ag1}\\
&\le& c(n,\ga)\|w\|_{A_3}\|f\|_{L^3(w)}^2.\nonumber
\end{eqnarray}

Let $E_j^k=Q_j^k\setminus \O_{k+1}$. It follows from the properties
(ii)-(iv) of Theorem \ref{main} that $|E_j^k|\ge |Q_j^k|/2$ and the
sets $E_j^k$ are pairwise disjoint. Hence, setting
$A_3(Q)=\frac{w(Q)(w^{-1/2}(Q))^2}{|Q|^3}$ (we use the notion
$\nu(Q)=\int_Q\nu(x)dx$), we have
\begin{eqnarray*}
&&\Big(\frac{1}{|\ga Q_j^k|}\int_{\ga Q_j^k}|f|\Big)^2\int_{Q_j^k}hw
\le 2(3\ga)^nA_3(3\ga Q_j^k)\\
&&\times\Big(\frac{1}{w^{-1/2}(3\ga Q_j^k)}\int_{\ga
Q_j^k}|f|\Big)^2\Big(\frac{1}{w(3\ga Q_j^k)}\int_{\ga
Q_j^k}hw\Big)|E_j^k|\\
&&\le
2(3\ga)^n\|w\|_{A_3}\int_{E_j^k}M^c_{w^{-1/2}}(fw^{1/2})^2M^c_wh\,dx
\end{eqnarray*}
(here $M^c_{\nu}f(x)=\sup_{Q\ni x}\frac{1}{\nu(Q)}\int_Q|f|\nu\,dx$,
where the supremum is taken over all cubes $Q$ centered at $x$).

Applying the latter estimate along with H\"older's inequality and
using the fact (based on the Besicovitch covering theorem) that the
$L^p(\nu)$-norm of $M^c_{\nu}$ does not depend on $\nu$, we get
\begin{eqnarray*}
&&\sum_{k,j}\Big(\frac{1}{|\ga Q_j^k|}\int_{\ga
Q_j^k}|f|\Big)^2\int_{Q_j^k}hw\\
&&\le 2(3\ga)^n\|w\|_{A_3}\int_{{\mathbb
R}^n}M^c_{w^{-1/2}}(fw^{1/2})^2M^c_wh\,dx\\
&&\le
2(3\ga)^n\|w\|_{A_3}\|M^c_{w^{-1/2}}(fw^{1/2})\|^2_{L^3(w^{-1/2})}\|M^c_wh\|_{L^3(w)}\\
&&\le c(n,\ga)\|w\|_{A_3}\|f\|_{L^3(w)}^2,
\end{eqnarray*}
and therefore the proof is complete.
\end{proof}

\subsection{Proof of Theorem \ref{LP}}
First, by (\ref{poest}) and Lemma~\ref{point}, it is enough to prove
(\ref{lp}) for $\widetilde G_{\a}$. Second, as we mentioned in the
Introduction, the sharp version of the Rubio de Francia
extrapolation theorem proved in \cite{DGPP} says that (\ref{lp}) for
$p=3$ implies (\ref{lp}) for any $p>1$. Therefore, our aim is to
show that
\begin{equation}\label{suff}
\|\widetilde G_{\a}(f)\|_{L^3(w)}\le
c(\a,n)\|w\|_{A_3}^{1/2}\|f\|_{L^3(w)}.
\end{equation}
Further, by a standard approximation argument, it suffices to prove
(\ref{suff}) for any $f\in L^1({\mathbb R}^n)$. By the weak type
$(1,1)$ property of $G_{\a}$ (\ref{weak}) and by (\ref{infty}) and
(\ref{poi}), for such $f$ we have
\begin{equation}\label{limit}
\lim_{|Q|\to \infty}|m_Q(\widetilde G_{\a}(f)^2)|\le
c\lim_{|Q|\to\infty}(G_{\a}f)^*(|Q|/2)^2=0.
\end{equation}

Now, following \cite{CMP}, denote by ${\mathbb R}^n_i, 1\le i\le
2^n$ the $n$-dimensional quadrants in ${\mathbb R}^n$, that is, the
sets $I^{\pm}\times I^{\pm}\times\dots\times I^{\pm}$, where
$I^+=[0,\infty)$ and $I^{-}=(-\infty,0)$. For each $i,1\le i\le
2^n,$ and for each $N>0$ let $Q_{N,i}$ be the dyadic cube adjacent
to the origin of side length $2^N$ that is contained in ${\mathbb
R}^n_i$. By (\ref{limit}) and by Fatou's lemma,
\begin{eqnarray}
&&\left(\int_{{\mathbb R}^n_i}\widetilde
G_{\a}(f)(x)^3w(x)dx\right)^{2/3}\nonumber\\
&&\le \liminf_{N\to\infty}\left(\int_{Q_{N,i}}|\widetilde
G_{\a}(f)(x)^2-m_{Q_{N,i}}(\widetilde
G_{\a}(f)^2)|^{3/2}w(x)dx\right)^{2/3}.\label{integ}
\end{eqnarray}

Combining Theorem \ref{main} (where $M_{1/4;Q_0}^{\#}f$ is replaced
by $M_{1/4;Q_0}^{\#,d}f$ from Remark \ref{rem}) with Lemma \ref{key}
(we use that the cubes $Q_j^k\in {\mathcal D}(Q_{N,i})$ are dyadic,
and hence the cubes $\widehat Q_j^k$ are dyadic as well; also,
$\widehat Q_j^k\subset 3Q_j^k$), we get that for all $x\in Q_{N,i}$,
\begin{equation}\label{inter}
|\widetilde G_{\a}(f)(x)^2-m_{Q_{N,i}}(\widetilde G_{\a}(f)^2)|\\
\le c_n\big(Mf(x)^2+{\mathcal A}_{45}f(x)\big),
\end{equation}
where
$$
Mf(x)=\sup_{Q\ni x}\frac{1}{|Q|}\int_Q|f(y)|dy
$$
is the Hardy-Littlewood maximal operator.

It was proved by Buckley \cite{B} that
\begin{equation}\label{buck}
\|M\|_{L^p(w)}\le
c(p,n)\|w\|_{A_p}^{\frac{1}{p-1}}\quad(1<p<\infty).
\end{equation}
Therefore,
$$
\left(\int_{Q_{N,i}}(Mf)^3w\right)^{2/3}\le
c(n)\|w\|_{A_3}\|f\|_{L^3(w)}^2.
$$
Applying this along with (\ref{integ}), (\ref{inter}) and Lemma
\ref{estag}, we get
$$
\left(\int_{{\mathbb R}^n_i}\widetilde
G_{\a}(f)(x)^3w(x)dx\right)^{2/3}\le
c(\a,n)\|w\|_{A_3}\|f\|_{L^3(w)}^2\quad(1\le i\le 2^n).
$$
Therefore,
\begin{eqnarray*}
\int_{{\mathbb R}^n}\widetilde
G_{\a}(f)(x)^3w(x)dx&=&\sum_{i=1}^{2^n} \int_{{\mathbb
R}^n_i}\widetilde G_{\a}(f)(x)^3w(x)dx\\
&\le& 2^n(c(\a,n)\|w\|_{A_3})^{3/2}\|f\|_{L^3(w)}^3,
\end{eqnarray*}
which completes the proof.

\subsection{Proof of Theorem \ref{SIN}} We use exactly the same
approach as in the proof of \cite[Corollary 1.4]{L2}. The proof is
just a combination of several known results.

First, it was proved by Wilson (\cite{Wi2} or \cite[p. 155]{Wi3})
that there exists $\a\le 1$ (depending on $T$) such that for all
$x\in {\mathbb R}^n$,
$$
S_{\psi}(Tf)(x)\le c(T,\psi,n)G_{\a}(f)(x).
$$
Exactly the same proof yields
\begin{equation}\label{poes}
S_{\psi,\b}(Tf)(x)\le c(T,\psi,n,\b)G_{\a}(f)(x)\quad (\b\ge 1).
\end{equation}

Next, define
$$\|w\|_{A_{\infty}}=\sup_Q\frac{1}{w(Q)}\int_QM(w\chi_Q)(x)dx.$$
It follows easily from (\ref{buck}) (see, e.g., \cite[Lemma
3.5]{L1}) that for any $p>1$,
\begin{equation}\label{ainf}
\|w\|_{A_{\infty}}\le c(p,n)\|w\|_{A_p}.
\end{equation}

Assuming that $\psi$ additionally satisfies
$$\int_s^{\infty}|\widehat \psi(t,0,\dots,0)|^2\frac{dt}{t}\ge
c(1+s)^{-\xi}$$ for some $s>0$, it was shown by Wilson \cite{Wi1}
that for any $p>0$,
\begin{equation}\label{GN}
\|{\mathcal M}(f)\|_{L^p(w)}\le
c(n,p,\psi)\|w\|_{A_{\infty}}^{1/2}\|S_{\psi,3\sqrt n}(f)\|_{L^p_w},
\end{equation}
where ${\mathcal M}$ is the grand maximal function. Note that
(\ref{GN}) is not contained in \cite{Wi1} in such an explicit form.
We refer to \cite[Proposition 2.3]{L2} for some comments about this.
Observe also that the proof of (\ref{GN}) is based essentially on
the deep theorem of Chang-Wilson-Wolff \cite{CWW}.

Further, it is well-known \cite[pp. 67-68]{St} that for all $x\in
{\mathbb R}^n$,
$$
T^*f(x)\le c(n,T)\big({\mathcal M}(Tf)+Mf(x)\big).
$$
Combining this with (\ref{poes}), (\ref{ainf}) and (\ref{GN}), we
get
\begin{eqnarray*}
\|T^*f\|_{L^p(w)}&\le& c\|w\|_{A_p}^{1/2}\|S_{\psi,3\sqrt
n}(Tf)\|_{L^p(w)}+c\|Mf\|_{L^p(w)}\\
&\le& c\|w\|_{A_p}^{1/2}\|G_{\a}(f)\|_{L^p(w)}+c\|Mf\|_{L^p(w)}.
\end{eqnarray*}
This estimate along with Theorem \ref{LP} and (\ref{buck}) for $p\ge
3$ yields
$$
\|T^*f\|_{L^p(w)}\le c\|w\|_{A_p}\|f\|_{L^p(w)},
$$
which completes the proof for $T^*$.

The above estimate for $T^*$ implies clearly the same estimate for
$T$:
$$
\|T\|_{L^p(w)}\le c\|w\|_{A_p}\quad(p\ge 3),
$$
which by duality yields
$$
\|T\|_{L^p(w)}\le
c\|w^{-\frac{1}{p-1}}\|_{A_{p'}}=c\|w\|_{A_p}^{\frac{1}{p-1}}\quad(1<p\le
3/2),
$$
and therefore, the theorem is proved.

\end{document}